\documentclass{article}

\usepackage[T1]{fontenc}
\usepackage[utf8]{inputenc}
\usepackage[english]{babel}
\usepackage{geometry}
\usepackage{amsmath}
\usepackage{amsthm}
\usepackage{color}
\usepackage{enumitem}
\usepackage[linktoc=all]{hyperref}				

\usepackage{mathtools}
\usepackage{amssymb}
\usepackage{wasysym}
\usepackage{textcomp}
\usepackage{stmaryrd}
\usepackage{esint}
\usepackage{xspace}

\newcommand*{\R}{\mathbb{R}}
\newcommand*{\C}{\mathbb{C}}
\newcommand*{\buc}[1]{BUC^{#1}}
\renewcommand*{\d}{\mathrm{d}}

\newcommand*{\deq}{\overset{\scriptscriptstyle{\mathrm{def}}}{=}}

\newcommand*{\re}{\ensuremath\mathrm{e}}
\newcommand*{\Norm}[1]{\left\lVert#1\right\rVert}
\newcommand*{\norm}[1]{\lVert#1\rVert}
\newcommand*{\Absolute}[1]{\left\lvert#1\right\rvert}
\newcommand*{\absolute}[1]{\lvert#1\rvert}

\renewcommand{\leq}{\leqslant}
\renewcommand{\geq}{\geqslant}
\DeclareMathOperator*{\real}{\mathrm{Re}}

\newtheorem{theorem}{Theorem}[section]

\newtheorem{proposition}[theorem]{Proposition}
\newtheorem{lemma}[theorem]{Lemma}
\theoremstyle{definition}
\newtheorem{definition}[theorem]{Definition}
\newtheorem{assumption}[theorem]{Assumption}

\theoremstyle{remark}
\newtheorem{remark}[theorem]{Remark}

\title{Degenerated scalar balance laws: stability of the critical wave}
\author{L. \textsc{Garénaux}}
\date{\today}

\begin{document}
	
\maketitle

\begin{abstract}
We consider scalar balance laws whith degenerated growth function at one infinity: $g'(u_+) = 0$. It is the transition point between monostable and bistable waves. We find it admits smooth profiles arising from both settings, and show stability for the slowest of these waves, under a sign assumption on the flux. The presence of a characteristic point at infinity weaken the temporal decay: an exponential spatial localization loss is necessary to get algebraic temporal decay.

\vspace{1em}
\noindent\textbf{AMS Classification subject:} 35B35, 35B40, 35C07, 35F20, 35F25 \newline
\end{abstract}

\section{Introduction}

\subsection{Overview}

We consider scalar first order PDE of the form
\begin{equation}
	\label{e:main}
	u_t + (f(u))_x = g(u).
\end{equation}
Here, $g$ and $f$ are elements of $\mathcal{C}^3(\R, \R)$. The flux term $f$ creates transport effects in the equation, while the growth term $g$ allow for growth or decay of solutions. 

Such equation have been recently studied in \cite{Duchene_Rodrigues_20,Duchene_Rodrigues_22,Garenaux_Rodrigues_25}. Existence and stability of traveling waves have been investigated, with a complete classification of stability mechanisms of bistable \cite{Duchene_Rodrigues_20,Duchene_Rodrigues_22} and monostable \cite{Garenaux_Rodrigues_25} waves. Interestingly, the stability of these waves is ruled by three type of points: infinities, discontinuity points and characteristic points.

The presence of a growth term recall the study of reaction-diffusion equation, which share a lot of similarity in term of solution behavior. Notably, existence of bistable waves for a unique speed and existence of monostable waves for all speed larger than a critical speed. 

In the present contribution, we study the case of a degenerate reaction term, meaning $g'(u(x)) \to 0$ as $x \to +\infty$. It corresponds to combustion fronts \cite{Ghazaryan_09a} in the reaction-diffusion case. 

This setting is the transition point between bistable and monostable waves. As such, we find that it admits both types of waves. A one parameter family of super-critical waves (all speeds larger than $c_*$, algebraic localization) as well as a critical speed (speed equal to $c_*$, exponential localization). We note that in the limit $g'(u_+) \to 0$ with $g'(u_+) < 0$, the bistable speed converges to the critical speed we study, and that the characteristic point goes to $+\infty$.

We further focus on the critical wave, and show its stability with respect to localized perturbation. The main technical feature of the stability proof is that at the critical speed, the transport coefficient vanishes at $x \to +\infty$.  This transport degenaracy (caused by a characteristic point at infinity) leads to poor weight stabilization: An exponential weight loss is only converted into polynomial temporal decay. This fact as several consequences: First, the spectrum can only be stabilized using doubly exponential weights $x \mapsto \re^{-\kappa_1 \re^{\kappa_2 x}}$, which we did not investigate here. Second, conjugation by wave like localized weight is necessary to remove growth $g'(u(x)) > 0$ near $x\to +\infty$, and does not stabilize the spectrum.

We do not investigate the stability of supercritical waves. Preliminary computations suggest that they are linearly stabilized by algebraic weights: $x\mapsto \sqrt{x}$ or $x\mapsto x^{-1/2}$ depending on $u_+ - u_-$ sign. The former is unbounded and can not be studied at nonlinear level. The latter is less localized than the wave itself, allowing for a large bassin of attraction.

Stability proof is done through evolution systems. See the discussion at the beginning of Section \ref{s:nonlinear_equation}.

\subsection{Assumptions, theorem statements}

We first assume the existence of two equilibrium points.
\begin{assumption}
	\label{a:properties_g}
	There exists $u_-$ and $u_+$ in $\R$ such that the following holds true.
	\begin{itemize}
	\item $g(u_-) = 0$ and $g(u_+) = 0$.
	\item $g'(u_-) < 0$ and $g'(u_+) = 0$, $g''(u_+) \neq 0$.
	\item $\forall u \in (u_-, u_+)$, $g(u) \neq 0$.
	\end{itemize}
\end{assumption}

We then look for solutions to \eqref{e:main} that connect these two points and travel at constant speed. 
\begin{assumption}
	\label{a:speed}
	Let $c\in \R$ such that for all $p \in [u_-, u_+)$, $c>f'(p)$. Assume either of the following two points.
	\begin{enumerate}
		\item \label{i:super_critical} (super-critical) $c > f'(u_+)$
		\item \label{i:critical} (critical) $c = f'(u_+)$ and $f''(u_+) \neq 0$ with sign given by $\frac{g''(u_+)}{f''(u_+)} < 0$.
	\end{enumerate}
\end{assumption}

\begin{remark}
	The case where $c = f'(u_*)$ with $u_* \in (u_-, u_+)$ is not considered here. $u_*$ is then a characteristic point, implying that characteristics would be ingoing the unstable infinity. Such pattern can not be stabilized and we do not study it.
\end{remark}

\begin{theorem}
	For every $c\in \R$ such that \ref{a:speed} holds, there exists $p\in \mathcal{C}^\infty(\R, (u_-, u_+))$ monotonic, such that 
	\begin{equation*}
		u(t, x) = p(x - ct)
	\end{equation*}
	solves \eqref{e:main}, and such that 
	\begin{equation*}
		\lim_{-\infty} p = u_-
		\hspace{4em}
		\lim_{+\infty} p = u_+.
	\end{equation*}
	If $c > f'(u_+)$, there exists $p_0 \in \R$ such that 
	\begin{equation*}
		p(x) = u_+ + \frac{\absolute{p_0}}{1 + \absolute{\frac{p_0}{2} \frac{g''(u_+)}{f'(u_+) - c}} x} + \mathcal{O}_{x\to +\infty}(\frac{1}{1 + x^2}).
	\end{equation*}
	If $c = f'(u_+)$, there exists $C_p \in \R$ and $\varepsilon>0$ such that 
	\begin{equation*}
		p(x) = u_+ + C_p \re^{-\frac{g''(u_+)}{2f''(u_+)} x} + \mathcal{O}_{x \to +\infty}(\re^{-\left(\Absolute{\frac{g''(u_+)}{2f''(u_+)}} + \varepsilon\right) x})
	\end{equation*}
\end{theorem}

Next, we investigate temporal stability of the critical wave. We need an additional assumption to ensure that $f$ does not contributes to growth of derivatives.

\begin{assumption}
	\label{a:sign_f_g}
	There exists $\eta>0$ and $x_0\in \R$ such that both following points hold.
	\begin{itemize}
		\item $\forall x\leq x_0$, $g'(p(x)) \leq - \eta + 2f''(p(x)) p'(x)$
		\item $\forall x \geq x_0$, $f''(p(x)) p'(x) > 0$
	\end{itemize}
\end{assumption}

\begin{remark}
	When $x_0$ is close enough to $+\infty$, the second point is automatically satisfied as will be seen in Lemma \ref{l:equilibria_ordering}. Unfortunately, we can not control the sign of $f''$ away from $u_+$. If $f''$ does not change sign on $(u_-, u_+)$, Assumption \ref{a:sign_f_g} is satisfied.
\end{remark}

\begin{definition}
	Let $\rho: \R \to [1, +\infty)$ be non-decreasing. If $\lim_{x\to +\infty} \rho(x) = +\infty$, we say that $\rho$ is a non-decreasing weight. If in addition, there exists positive constants $C$ and $\kappa$ such that for all $x \geq 0$, 
	\begin{equation*}
		\rho(x) \leq C \re^{\kappa x}
	\end{equation*} 
	then we say that $\rho$ is a weight controlled by $\re^{\kappa x}$, or that $\rho$ is sub-exponential with constants $(C, \kappa)$.
\end{definition}

\begin{theorem}
	\label{t:stability}
	Assume assumptions \ref{a:properties_g}, \ref{a:speed} and \ref{a:sign_f_g}. Then for $c = f'(u_+)$, the profile $p$ is asymptotically stable in exponentially weighted norms, against algebraically localized initial data: 
	There exists $C>0$, $\delta>0$ and a weight function $\omega: \R \to (0,1]$ satisfying
	\begin{equation*}
		\omega(x) \sim_{x\to +\infty} C_\omega \re^{-\Absolute{\frac{g''(u_+)}{2 f''(u_+)}} y}
	\end{equation*}
	such that for every $w_0 \in \buc{1}(\R)$, for all $\alpha>0$, if $\norm{w_0} \leq \delta$ then there exists a unique solution to \eqref{e:main} with initial data $u_0 = p + \omega w_0$ and it satisfies for all $t\geq 0$
	\begin{equation*}
		\Norm{\frac{u(t) - p(x - ct)}{\omega}} \leq \frac{C^2}{\big(C + \ln(1 + \frac{1}{C} t)\big)^\alpha} \left(\norm{(1 + \cdot)^\alpha w_0}_{W^{1, \infty}([0,+\infty))} + \norm{w_0}_{W^{1, \infty}((-\infty, 0])}\right)
	\end{equation*}
\end{theorem}

\section{Existence}

Looking for a traveling wave solution, we insert the ansatz $u(t,x) = p(x - ct)$ into \eqref{e:main} and obtain
\begin{equation}
	\label{e:profile}
	(-c + f'(p)) p' = g(p).
\end{equation}

\begin{lemma}
	Let $c$ satisfying Assumption \ref{a:speed}. There exists a solution $p\in \mathcal{C}^\infty(\R, (u_-, u_+))$ to \eqref{e:profile}. It satisfies
	\begin{equation*}
		\lim_{x\to -\infty} p(x) = u_-, 
		\hspace{4em}
		\lim_{x\to +\infty} p(x) = u_+. 
	\end{equation*}
\end{lemma}
\begin{proof}
	From Assumption \ref{a:speed}, the map $F: [u_-, u_+) \to \R$ 
	\begin{equation*}
		F(p) = \frac{g(p)}{f'(p) - c}
	\end{equation*}
	is well defined, $\mathcal{C}^1$, and has constant sign. 
	When either items of Assumption \ref{a:speed} is satisfied, $F$ can be extended to $[u_-, u_+]$ with $F(u_+) = 0$. The extended map is $\mathcal{C}^1$, and satisfies $F(u_-) = 0$, $F(u_+) = 0$.
	
	Now that we have collected a few fact about $F$, remark that the profile equation \eqref{e:profile} rewrites $p' = F(p)$. Thus, this ODE admits two constant solutions $p = u_-$ and $p = u_+$. Let us fix $p(0) = \frac{u_- + u_+}{2}$. The profile equation being an scalar first order ODE, non intersection of solutions guarantees that the corresponding solution satisfies $p(x) \in (u_-, u_+)$ for all $x\in \R$. Due to $p$ being monotone and bounded, it converges at both infinities. The only possible limits are $u_-$ and $u_+$ due to $F$ non-vanishing on the interval interior.
	The proof is complete. We incidentaly note that 
	\begin{equation*}
		F(p) = \frac{g''(u_+)}{2 f''(u_+)} (p-u_+) + \mathcal{O}_{p \to u_+}((p-u_+)^2).
	\end{equation*}
	As a consequence, there exists a constant $C_p$ such that when $x\to +\infty$:
	\begin{equation}
		\label{e:profile_asymptotics}
		\absolute{p(x) - u_+} \sim C_p \re^{-\kappa x}
	\end{equation}
	where 
	\begin{equation*}
		\kappa \deq \Absolute{\frac{g''(u_+)}{2 f''(u_+)}} > 0.
	\end{equation*}
\end{proof}

\begin{lemma}
	\label{l:equilibria_ordering}
	If $u_- \leq u_+$, then the following signs hold true: $g''(u_+) < 0$, $f''(u_+) > 0$, $p - u_+ < 0$ and $p'<0$.
	If $u_+ \leq u_-$, then the following signs hold true: $g''(u_+) > 0$, $f''(u_+) < 0$, $p - u_+ > 0$ and $p'>0$.
\end{lemma} 
\begin{proof}
	The sign of $g''(u_+)$ is deduced from the facts that $g$ does not vanish and that $g'(u_-) < 0$. The other three signs are easily obtained.
\end{proof}

\section{Linear stability}

Our basic tool to control temporal growth/decay is a resolvent estimate from \cite{Duchene_Rodrigues_20}. We present it in section \S \ref{s:linear_decay_1}. We can also improve decay rates in exchange of spatial localization. We detail it in section \ref{s:linear_decay_2}

\subsection{Temporal decay without localization loss}
\label{s:linear_decay_1}

We consider the following generic eigenproblem:
\begin{equation}
	\label{e:spectral_pb}
	\lambda v - a \partial_x v - b = f.
\end{equation}

\begin{proposition}
	\label{p:resolvent_estimates}
	Let $a, b \in \buc{0}(\R)$, and let $f \in \buc{0}(\R, \R)$. For every $\lambda \in \C$ such that 
	\begin{equation*}
		\real{\lambda} \geq \sup_\R b,
	\end{equation*}
	there exists a unique $v \in \buc{1}(\R, \R)$ satisfying \eqref{e:spectral_pb}. Furthermore, 
	\begin{equation*}
		\norm{v}_{W^{1, \infty}(\R)} \leq \frac{\norm{f}_{L^\infty(\R)}}{\sup_\R b - \real{\lambda}}.
	\end{equation*}
\end{proposition}
\begin{proof}
	See \cite{Duchene_Rodrigues_20}.
\end{proof}

\begin{proposition}
	\label{p:generic_linear_decay}
	Let $a, b \in \buc{0}([0,+\infty) \times \R, \R)$ and $u_0 \in \buc{1}(\R, \R)$. There exists a unique solution $v\in \buc{0}([0,+\infty), \buc{1}(\R, \R))$ to the linear problem
	\begin{equation}
		\label{e:linear_problem}
		\begin{cases*}
			v_t + a v_x = b v, \\
			v(0, x) = v_0(x)
		\end{cases*}
	\end{equation}
	Futhermore, it satisfies for all $t \geq 0$
	\begin{equation*}
		\norm{v(t)}_{W^{1, \infty}(\R)} \leq \exp\left(\int_0^t \sup_\R b(s, \cdot) \d s\right) \norm{v_0}_{W^{1, \infty}(\R)}
	\end{equation*}
\end{proposition}
\begin{proof}
	It is the same proof as in \cite{Duchene_Rodrigues_20}. From the resolvent estimates with constant one (see Lemma \ref{p:resolvent_estimates}) standard theory on evolution systems \cite{Pazy_83} gives existence and uniqueness of a solution. It satisfies the claimed estimate.
\end{proof}

\subsection{Temporal decay with localization loss}
\label{s:linear_decay_2}

Before stating and proving the main statement of this section, let us recall some well known fact about transport equation.

\begin{equation}
	\label{e:ODE_a}
	\begin{cases}
		\partial_t X(t,x) = a(t, X(t, x))\\
		X(0, x) = x
	\end{cases}
\end{equation}

\begin{lemma}
	Assume that $a$ is bounded from above and below. Then for all $t\geq 0$ there exist $X_\mathrm{inv}(t) : \R \to \R$ such that for all $y \in \R$,
	\begin{equation*}
		X(t, X_\mathrm{inv}(t, y)) = y.
	\end{equation*}
\end{lemma}
\begin{proof}
	We have to prove surjectivity of the maps $x\mapsto X(t,x)$, $t\geq 0$. Let $y \in \R$. For all $t\geq 0$ let $a_\mathrm{min}(t)$, $a_\mathrm{max}(t)$ such that for all $x\in \R$
	\begin{equation*}
		a_\mathrm{min}(t) \leq a(t, x) \leq a_\mathrm{max}(t).
	\end{equation*}
	Let $t\geq 0$ fixed, and remark that 
	\begin{equation*}
		X(t, x) \in \left[x + \int_0^t a_\mathrm{min}(s), x + \int_0^t a_\mathrm{max}(s)\right].
	\end{equation*}
	As a consequence, taking $x \to -\infty$ provides an $x_1(t)$ such that $X(t, x_1) \leq x_1 + \int_0^t a_\mathrm{max}(s) \d s < y$. Similarly there exists $x_2(t)$ such that $X(t, x_2) \geq x_2 + \int_0^t a_\mathrm{min}(s) \d s > y$. By intermediate values theorem ($X(t, \cdot)$ is continuous since solution to the ODE \eqref{e:ODE_a} are smooth with respect to variation of the inial data) there exists $x_0 \in [x_1, x_2]$ such that $X(t, x_0) = y$. The proof is complete.
\end{proof}

\begin{proposition}
	\label{p:linear_decay_with_weight_loss}
	Let $a, b\in \buc{0}([0,+\infty) \times \R, \R)$ such that there exists positive constants $C_a$, $\kappa_a$, $a_0$, $\eta$ satisfying $\forall t \geq 0$, $\forall x \geq 0$,
	\begin{equation}
		\label{e:hypothesis_linear_coefficients}
		a(t,x) \leq -C_a \re^{-\kappa_a x}, 
		\hspace{4em}
		b(t,x) \leq 0
	\end{equation}
	and $\forall x\leq 0$, 
	\begin{equation}
		\label{e:hypothesis_linear_coefficients_2}
		a(t,x) \leq -a_0, 
		\hspace{4em}
		b(t,x) \leq -\eta
	\end{equation}
	Then there exists $C>0$ such that for every $v_0 \in \buc{0}(\R, \R)$, and every non-decreasing weight $\rho:\R \to [1, +\infty)$  the unique solution to \eqref{e:linear_problem} satisfies
	\begin{equation*}
		\norm{v(t,\cdot)}_{L^\infty(\R)} \leq \frac{1}{\rho\left(\frac{\ln (1 + C_a \kappa_a t)}{\kappa_a}\right)} \norm{\rho v_0}_{L^\infty(\R)}.
	\end{equation*}
	If in addition there exists $C_b > 0$ such that for all $t \geq 0$ and $\forall x \in \geq 0$
	\begin{equation*}
		b - a_x (t,x) \leq 0, 
		\hspace{4em}
		b_x (t,x) \leq - C_b \re^{-\kappa_a x},
	\end{equation*}
	while $\forall x \leq 0$
	\begin{equation*}
		b - a_x (t,x) \leq -\eta,
	\end{equation*}
	then there exists $C>0$ such that for every $v_0 \in \buc{0}(\R, \R)$, for every $\alpha > 0$ for all $t\geq 0$
	\begin{equation*}
		\norm{v(t,\cdot)}_{W^{1, \infty}(\R)} \leq \frac{C}{(1 + C_a \kappa_a t)^\alpha} \norm{(1 + \re^{\kappa_a \alpha \cdot}) v_0}_{W^{1,\infty}(\R)}.
	\end{equation*}
\end{proposition}
\begin{proof}
	The computation 
	\begin{equation*}
		\frac{\d }{\d t} (v(t, X(t, x))) = v_t + a v_x = b(t, X(t, x)) v(t, X(t, x))
	\end{equation*}
	leads to 
	\begin{equation*}
		v(t, X(t, x)) = \re^{\int_0^t b(s, X(s, x)) \d s} v_0(x).
	\end{equation*}
	At the special value $x = X_\mathrm{inv}(t, y)$, it reads
	\begin{equation*}
		v(t, y) = \re^{\int_0^t b(s, X(s, X_\mathrm{inv}(t, y))) \d s} v_0(X_\mathrm{inv}(t, y)).
	\end{equation*}
	Multiplying and dividing by $\rho$, we obtain
	\begin{align*}
		\absolute{v(t, y)} \leq {}&{} \frac{1}{\rho(X_\mathrm{inv}(t, y))} \absolute{\rho(X_\mathrm{inv}(t, y)) v_0(X_\mathrm{inv}(t, y))},\\
		\leq {}&{} \frac{1}{\rho(X_\mathrm{inv}(t, y))} \norm{\rho v_0}_{L^\infty(\R)}.
	\end{align*}
	To further improve this estimate, we control $X_\mathrm{inv}$ from below. We rely on $a$ estimate:
	\begin{equation*}
		\partial_t X(t, x) \leq -C_a \re^{-\kappa_a X(t, x)}
	\end{equation*}	
	Multiplying both side by an exponential we see the time derivative of an exponential. Integrating reads
	\begin{equation*}
		\re^{\kappa_a X(t, x)} - \re^{\kappa_a x} \leq - C_a \kappa_a t
	\end{equation*}
	Inverting this formula leads to 
	\begin{equation*}
		x \geq \frac{1}{\kappa_a} \ln(\re^{\kappa_a X(t,x)} + C_a \kappa_a t),
	\end{equation*}
	in particular with $x = X_\mathrm{inv}(t, y)$ we obtain for $y\geq 0$
	\begin{align*}
		X_\mathrm{inv}(t, y) \geq {}&{} \frac{1}{\kappa_a} \ln (\re^{\kappa_a y} + C_a \kappa_a t)\\
		\geq {}&{} \frac{1}{\kappa_a} \ln (1 + C_a \kappa_a t).
	\end{align*}
	Since $\rho$ is non-decreasing, the two estimates combine and give the claimed $L^\infty([0, +\infty))$ estimate. 
	
	We still need to control $v$ on $(-\infty, 0]$. There, the space time domain $(-\infty, 0] \times [0, +\infty)$ is decomposed into two parts: The region where the solution originate from the initial data $v_0$ and the region where it originate from the boundary condition $x = 0$. In the latter, $b\leq 0$ guarantees no growth, and the boundary value itself decays from the previous discussion. In the former, $b\leq -\eta$ guarantees exponential temporal decay, which is controled by $\frac{1}{\rho}$.
		
	To control the first space derivative, we differentiate \eqref{e:linear_problem}.
	\begin{equation*}
		(v_x)_t + a (v_x)_x = (b - a_x) v_x + b_x v.
	\end{equation*}
	Denoting $\mathcal{S}_{a,b}(t,0) v_0$ the solution to \eqref{e:linear_problem}, the variational formula writes
	\begin{equation*}
		v_x = \mathcal{S}_{a, b-a_x}(t,0) v_0 + \int_0^t \mathcal{S}_{a, b-a_x}(t,s) (b_x(s) v(s)) \d s.
	\end{equation*}
	Since $b - a_x \leq 0$, hypothesis \eqref{e:hypothesis_linear_coefficients} is satisfied by $a$ and $b - a_x$. We can apply the above proof to get 
	\begin{equation*}
		\norm{\mathcal{S}_{a, b-a_x} (t,0) v_0}_{L^\infty(\R)} \leq \frac{1}{\rho\left(\frac{\ln (1 + C_a \kappa_a t)}{\kappa_a}\right)} \norm{\rho v_0}_{L^\infty(\R)}
	\end{equation*}
	allowing to estimate the first term in the right hand side of the variational formula. 
	To estimate the second term, we rely the above discussion. For any sub-exponential weight $\rho$
	\begin{align*}
		\int_0^t \norm{\mathcal{S}_{a, b-a_x}(t,s) (b_x(s) v(s))}_{L^\infty([0,+\infty))} \d s \leq {}&{} \int_0^{t} b_x\left(\frac{1}{\kappa_a} \ln(1 + C_a \kappa_a (t-s))\right) \norm{v(s)}_{\R} \d s \\
		\leq {}&{} \int_0^{t} \frac{C_b}{1 + \kappa_a C_a (t-s)} \frac{\norm{\rho v_0}_{L^\infty(\R)}}{\rho(\frac{1}{\kappa}\ln(1 + C_a b_a s))} \d s.
	\end{align*}
	Specifying to $\rho = \re^{\kappa_a \alpha \cdot}$ and using a classical formula involving the Gamma function, we obtain
	\begin{align*}
		\int_0^t \norm{\mathcal{S}_{a, b-a_x}(t,s) (b_x(s) v(s))}_{L^\infty([0,+\infty))} \d s \leq {}&{} \int_0^{t} \frac{C_b}{1 + \kappa_a C_a (t-s)} \frac{\norm{(1 + \re^{\kappa_a \alpha \cdot}) v_0}_{L^{\infty}(\R)}}{(1 + \kappa_a C_a s)^\alpha} \d s \\
		\leq {}&{} C \frac{\norm{(1 + \re^{\kappa_a \alpha \cdot}) v_0}_{L^{\infty}(\R)}}{(1 + \kappa_a C_a t)^\alpha}.
	\end{align*}
	We estimated both terms of the variational formula, the proof is complete.
\end{proof}

\section{Nonlinear equation}
\label{s:nonlinear_equation}

\begin{definition}
	We use the following notation for second order Taylor expansions:
	\begin{equation*}
		T_{2, F}(a, b) = F(a + b) - F(a) - F'(a)b.
	\end{equation*}
\end{definition}

Design of the nonlinear iteration scheme is subtle for this problem. Inserting a perturbative ansatz $u(t, x) = p(x - ct) + v(t, x - ct)$ into \eqref{e:main}, we obtain the perturbation dynamic. The first approach is to write it as 
\begin{equation}
	\label{e:unweighted_dynamics}
	v_t = \mathcal{A}(x) v + N_1(x,v)
\end{equation}
with coefficients and nonlinearity defined as follow.
\begin{align*}
	\mathcal{A}(x) = {}&{} -(f'(p) - c) \partial_x + g'(p) - p'f''(p)\\
	N_1(x,v) = {}&{} T_{2, g}(p(x), v) - \left(T_{2, f}(p(x), v)\right)_x
\end{align*}
Unfortunately, the estimate $\norm{N_1(v)}_{W^{1, \infty}} \leq C \norm{v}_{W^{1, \infty}} \norm{v}_{W^{2, \infty}}$ prevent to close the argument.

Instead, we consider a $v$ dependent transport coefficients. Only terms of the form $a(v) \partial_x v$ can be treated this way. It reads
\begin{equation*}
	v_t + \big(f'(p+v) - c\big)v_x = (g'(p) - f''(p))v + N_2(x, v)
\end{equation*}
with 
\begin{equation}
	\label{e:def_N_2}
	N_2(\cdot, v) \deq T_{2,g}(p,v) - T_{2,f'}(p,v) p'
\end{equation}

\subsection{Weight construction}

In order to apply linear estimates, we first need to remove the temporal growth induced by linear terms. The present discussion is made more complex by the fact that transport coefficient is $v$-dependent. We start with a simplified statement that will not be used in the main proof.

\begin{lemma}
	There exists $\omega:[0,+\infty) \times \R \to (0,1]$ and positive constants $a_0$, $\eta$, such that 
	\begin{equation*}
		\frac{1}{\omega(t,x)} \mathcal{A}(t,x) (\omega(t,x) \cdot) \deq -a(t,x) \partial_x + b(t,x)
	\end{equation*}
	satisfies for all $t \geq 0$ and $x \geq 0$:
	\begin{align*}
		a(x) \leq -\absolute{f''(u_+)} \frac{C_p}{2} \re^{-\kappa x} 
		\hspace{4em}
		b(x) \leq 0,
	\end{align*}
	while $\forall x \leq 0$
	\begin{align*}
		a(x) \leq -a_0 
		\hspace{4em}
		b(x) \leq -\eta,
	\end{align*}
	Furthermore for all $x \geq 0$
	\begin{align*}
		b - a_x (t,x) \leq 0, 
		\hspace{4em}
		b_x (t,x) \leq - C_b \re^{- \kappa x},
	\end{align*}
	while $\forall x \leq 0$
	\begin{align*}
		b - a_x (t,x) \leq - \eta, 
	\end{align*}
	Furthermore, there exists $C_\omega > 0$ such that 
	\begin{equation*}
		\omega(x) = (C_\omega + \mathrm{o}_{x \to +\infty}(1)) \re^{-\kappa x}, 
		\hspace{4em}
		\omega(x) \longrightarrow_{x\to -\infty} 1.
	\end{equation*}
\end{lemma}
\begin{proof}
	Let 
	\begin{equation*}
		\omega(x) \deq \exp\left(-\int_{\infty}^x \frac{\big(g'(p(y)) - p'(y)f''(p(y))\big)_+}{c - f'(p(y))} \d y \right),
	\end{equation*}
	with $\cdot_+ = \frac{\cdot + \absolute{\cdot}}{2}$ denoting the positive part.
	Remark first that $g'(p(y)) - p'(y)f''(p(y) \longrightarrow g'(u_-) < 0$ when $y\to -\infty$. Thus the numerator is zero for all $y$ close enough to $-\infty$, and the integral converges. Incidentally $\omega(y) = 1$ for these values of $y$, and the limit $y\to -\infty$ is proven.
	
	For this choice of $\omega$, 
	\begin{equation*}
		\frac{1}{\omega(t,x)} \mathcal{A}(t,x) (\omega(t,x) \cdot) = -(f'(p(x)) - c) \partial_x - \big(g'(p(x)) - p'(x)f''(p(x))\big)_-.
	\end{equation*}
	First, $b \leq 0$ is satisfied, and  
	\begin{equation*}
		\frac{d}{d_x}\left(g'(p) - p'f''(p)\right) = p' g''(p) - p'' f''(p) - {p'}^2 f'''(p) = \mathcal{O}_{x \to +\infty}(\re^{-\kappa x}).
	\end{equation*}
	Second, assuming \ref{i:critical} we compute
	\begin{align*}
		a = {}&{} f'(p) - c, \\
		= {}&{} f''(u_+)(p - u_+) + \mathcal{O}_{p \to u_+}\big((p - u_+)^2\big), \\
		= {}&{} - \absolute{f''(u_+)} C_p \re^{-\kappa x} + \mathcal{O}_{x\to +\infty}(\re^{-2 \kappa x}).
	\end{align*}
	which leads to the claim on $a$. The sign of $b - a_x$ follows from assumption \ref{a:sign_f_g}. Remark that with $\eta\geq 0$, one has $-(\alpha - \beta)_- - \beta \leq -\eta$ when $\beta \geq \eta$ or when $\alpha - 2\beta \leq \eta$.
	
	We then estimate $\omega$ close to $+\infty$. We start with the numerator:
	\begin{align*}
		g'(p) - p'f''(p) = {}&{} g''(u_+) (p - u_+) + \mathcal{O}_{y\to +\infty}((p-u_+)^2)- F(p) f''(p), \\
		= {}&{} (g''(u_+) - \kappa f''(u_+)) (p - u_+) + \mathcal{O}_{y\to +\infty}((p - u_+)^2), \\
		= {}&{} \frac{g''(u_+)}{2} (p - u_+) + \mathcal{O}_{y\to +\infty}((p - u_+)^2),
	\end{align*}
	and continue with the denominator:
	\begin{equation*}
		c - f'(p) = f''(u_+) (p - u_+) + \mathcal{O}_{y\to +\infty}((p-u_+)^2).
	\end{equation*}
	Combined together, we compute
	\begin{equation}
		\label{e:weight_expansion}
		\frac{g'(p) - p'f''(p)}{c - f'(p)} = \frac{g''(u_+)}{2 f''(u_+)} \frac{1 + \mathcal{O}_{y\to +\infty}((p - u_+)^2)}{1 + \mathcal{O}_{y\to +\infty}((p - u_+)^2)}.
	\end{equation}
	Since the numerator is positive when $y$ is close enough to $+\infty$, adding a positive part is harmless to the expansion. It provides the claim after a straightforward integration.
\end{proof}

The correct strategy is to find a weight $\omega$ that stabilizes all linear dynamics, when $v$ runs over a small ball of $\buc{1}(\R, \R)$.

\begin{lemma}
	\label{l:def_weight}
	There exists $\omega:[0,+\infty) \times \R \to (0,1]$ and $\delta>0$ such that for every $v\in \buc{1}(\R)$ satisfying $\norm{v}_{W^{1, \infty}} \leq \delta$, the coefficients 
	\begin{equation*}
		\frac{1}{\omega} \big(c - f'(p + \omega v)\big) \partial_x \big( \omega v \big) + (g'(p) - f''(p)) \deq - a(x, v) \partial_x v + b(x, v)
	\end{equation*}
	satisfy
	\begin{equation*}
		a(x, v) \leq -\absolute{f''(u_+)} \frac{C_p}{2} \re^{-\kappa x} 
		\hspace{4em}
		b(x, v) \leq 0.
	\end{equation*}
	Furthermore, there exists $C_\omega > 0$ such that 
	\begin{equation*}
		\omega(x) = (C_\omega + \mathrm{o}_{x \to +\infty}(1)) \re^{-\kappa x}, 
		\hspace{4em}
		\omega(x) \longrightarrow_{x\to -\infty} 1.
	\end{equation*}
	Furthermore for all $x \geq 0$
	\begin{align*}
		b - a_x (t,x) \leq 0, 
		\hspace{4em}
		b_x (t,x) \leq - C_b \re^{- 2\kappa x},
	\end{align*}
	while $\forall x \leq 0$
	\begin{align*}
		b - a_x (t,x) \leq - \eta, 
	\end{align*}
\end{lemma}
\begin{proof}
	Let $C_1>0$, and 
	\begin{equation*}
		\omega(x) \deq \exp\left(\int_{-\infty}^x \frac{\big(g'(p) - f''(p)\big)_+}{\sup_{\absolute{v} \leq \delta C_1} \absolute{c - f'(u_+ + (1 + v)(p - u_+))}} \d y \right).
	\end{equation*}	
	Assuming $\delta$ is taken small enought, the expansion \eqref{e:weight_expansion} is easily adapted to the present definition with minor changes. As a consequence, $\omega(x) \leq C_2 \re^{-\kappa x}$ still holds true for a constant $C_2$ that does not change when $\delta$ is taken smaller. In particular, there exists a constant $C_{\omega, p}$ such that $\forall x \in \R$
	\begin{equation}
		\label{e:estimate_omega_p}
		\omega(x) \leq C_{\omega, p} \absolute{p - u_+}. 
	\end{equation}
	Up to setting $\delta$ to a smaller value, we can replace $C_1$ by $C_{\omega, p}$ in the definition of $\omega$:
	\begin{equation*}
		\omega(x) \deq \exp\left(\int_{-\infty}^x \frac{\big(g'(p) - f''(p)\big)_+}{\sup_{\absolute{v} \leq \delta C_{\omega, p}} \absolute{c - f'(u_+ + (1 + v)(p - u_+))}} \d y \right),
	\end{equation*}	
	and \eqref{e:estimate_omega_p} still holds true with this definition.
	
	With the statement notation, we compute
	\begin{equation*}
		a(x, v) = c - f'(p + \omega v), 
		\hspace{4em}
		b(x, v) = g'(p) - f''(p) + \frac{\omega'}{\omega} a(x, v).
	\end{equation*}
	from $v\in \buc{1}(\R)$ and \eqref{e:estimate_omega_p}, the claimed estimate on $a$ is true. To estimate $b$, we discuss the sign of $u_+ - u_-$. When positive, $f'$ is increasing on $(u_-, u_+)$. Assuming $\delta$ small enough we ensure that $p + \omega v \in (u_-, u_+)$. As a consequence, we can estimate for large enough $x$:
	\begin{equation*}
		\sup_{\absolute{v} \leq C_{\omega, p} \delta} \absolute{c - f'(u_+ + (1 + v)(p - u_+))} \leq c - f'(u_+ + (1 + \delta C_{\omega, p})(p - u_+)).	
	\end{equation*} 
	On the other hand, 
	\begin{equation*}
		c - f'(p + \omega v) \leq c - f'(u_+ p - u_+ + \delta C_{\omega, p}(p - u_+)).	
	\end{equation*} 
	The two estimate together lead to $b(x) \leq - \big(g'(p) - f''(p)\big)_- \leq 0$. When $u_+ - u_-$ is negative, the same proof structure applies with a sign change. The intermediate quantity to reach being $c - f'(u_+ + (1 - \delta C_{\omega, p})(p - u_+))$ due to $f'$ being decreasing.
	
	The claim on $b_x$ is proven with a similar Taylor expansion as in the $v$-independent case. The claim on $b - a_x$ is a combination of both the above argument to control $b$ and the $v$-independent estimate on $b - a_x$.
\end{proof}

\subsection{Weighted dynamics}

We further restrict to perturbations $v$ of the form 
\begin{equation*}
	v = \omega w
\end{equation*}
with $w \in \buc{1}(\R)$. Inserting this ansatz into \eqref{e:unweighted_dynamics}, we obtain

\begin{equation*}
	w_t + a \partial_x w = b w + \mathcal{N}(\cdot, w),
\end{equation*}
where $a$ and $b$ are defined in Lemma \ref{l:def_weight}, and 
\begin{equation*}
	\mathcal{N}(x, w) = \frac{1}{\omega(x)} N_2(x, \omega(x) w(x)).
\end{equation*}

\begin{lemma}
	There exists $C>0$ and $\delta>0$ such that for every $w\in \buc{1}(\R)$ with $\norm{w}_{W^{1, \infty}(\R)} \leq \delta$, 
	\begin{equation*}
		\Norm{\frac{1}{\omega} \mathcal{N}(\cdot, w)} \leq C \norm{w}_{W^{1, \infty}(\R)}.
	\end{equation*}
\end{lemma}
\begin{proof}
	Direct from $N_2$ expression \eqref{e:def_N_2}.
\end{proof}

To prove temporal decay of the weighted perturbation, we introduce the shape function
\begin{equation*}
	\Theta(t) = \sup_{s\in [0, t]} (1 + s)^\alpha \norm{(1 + \re^{\kappa \alpha \cdot}) w(s)}_{W^{1, \infty}(\R)}
\end{equation*}

By short time existence theory \cite{Kruzkov_70}, $\Theta$ is continuous. We now prove it is bounded from above when the initial data is small enough.

\begin{proposition}
	Let $\alpha > 0$ and $E_0 = \norm{(1 + \re^{\kappa \alpha \cdot}) w_0}$. There exists $C>0$ such that 
	\begin{equation}
		\label{e:Theta_estimate}
		\Theta(t) \leq C E_0 + C \Theta(t)^2.
	\end{equation}
\end{proposition}
\begin{proof}
	We rely on the Duhamel formula
	\begin{equation*}
		w(t) = \mathcal{S}_{a, b}(t, 0) w_0 + \int_0^t \mathcal{S}_{a, b}(t,s) \mathcal{N}(\cdot, w(s)) \d s.
	\end{equation*}
	Applying proposition \ref{p:linear_decay_with_weight_loss}, we obtain
	\begin{align*}
		\norm{w(t)}_{W^{1, \infty}(\R)} \leq {}&{} \frac{C E_0}{(1 + \kappa C_a t)^\alpha} + C \int_0^t \frac{1}{1 + \kappa C_a (t - s)} \norm{\frac{1}{\omega}\mathcal{N}(w)}_{W^{1, \infty}} \d s \\
		\leq {}&{} \frac{C E_0}{(1 + \kappa C_a t)^\alpha} + C \int_0^t \frac{1}{1 + \kappa C_a (t - s)} \frac{\Theta(s)^2}{(1 + \kappa C_a s)^{2\alpha}} \d s \\
		\leq {}&{} \frac{C E_0}{(1 + \kappa C_a t)^\alpha} + \Theta(t)^2 \int_0^t \frac{1}{(1 + \kappa C_a (t - s))^{1 - \alpha}} \frac{1}{(1 + \kappa C_a s)^{2\alpha}} \d s.
	\end{align*}
	Using the same Gamma function formula as previously, we obtain
	\begin{equation*}
		(1 + \kappa C_a t) \norm{w(t)}_{W^{1, \infty}(\R)} \leq C E_0 + \Theta(t)^2.
	\end{equation*}
	Which leads to the claim and complete the proof.
\end{proof}

\begin{proof}[Proof -- Theorem \ref{t:stability}]
	From inequality \eqref{e:Theta_estimate}, it is standard to deduce that $\Theta(t) \leq 2 C E_0$ when $E_0$ is small enough. We describe the main line of the argument for completeness. Consider the map $\gamma(x) = CE_0 + Cx^2 - x$. Inequality \eqref{e:Theta_estimate} states that $\gamma(\Theta)$ is non-negative. It is also continuous. A quick sign study of $\gamma$ reveals that it is positive on a region of the form $[0, d_1] \cup [d_2, +\infty)$
	with $d_1 < 2CE_0 < d_2$. Assuming $\Theta(0) \leq 2C E_0$, we deduce that $\Theta(0) \leq d_1$, and that $\Theta(t) \leq d_1$ for all times. Otherwise, there would exists $t_0>0$ such that $\gamma(\Theta(t_0)) = \gamma(2C E_0) < 0$.
\end{proof}

\bibliographystyle{alpha}
\bibliography{./critical_balance_laws.bib}

\end{document}